\documentclass[10pt,a4paper]{amsart}
\usepackage{amssymb,amsmath,amsfonts}
\headsep=1.2500cm \numberwithin{equation}{section}
\newcommand{\A}{UP(I,\mathcal{A})}
\newcommand{\abs}[1]{\left\vert#1\right\vert}
\newcommand{\set}[1]{\left\{#1\right\}}
\newcommand{\norm}[1]{\left\Vert#1\right\Vert}
\newcommand{\ls}{\ell^{1}\mbox{-}\underset{i\in{I}}{\oplus}\mathcal{A}_{i}}
\newcommand{\ljs}{\ell^{1}\mbox{-}\underset{\alpha\in{J}}{\oplus}\mathcal{A}_{\alpha}}
\newcommand{\cs}{c_{0}\mbox{-}\underset{i\in{I}}{\oplus}\mathcal{A}_{i_\ast}}
\newcommand{\ccs}{c_{00}\mbox{-}\underset{i\in{I}}{\oplus}\mathcal{A}_{i_\ast}}
\newcommand{\cjs}{c_{0}\mbox{-}\underset{\alpha\in{J}}{\oplus}\mathcal{A}_{\alpha_\ast}}
\newcommand{\linf}{\ell^\infty\mbox{-}\underset{\alpha\in{J}}{\oplus}\mathcal{A}_
		{\alpha}^{\ast}}

\newtheorem{Theorem}{Theorem}[section]
\newtheorem{Proposition}[Theorem]{Proposition}

\newtheorem{lemma}[Theorem]{Lemma}
\theoremstyle{remark}

\newtheorem{Example}[Theorem]{Example}

\newtheorem{Remark}[Theorem]{Remark}

\begin{document}

\title{On Connes amenability of upper triangular matrix algebras}

\author[S. F. Shariati]{S. F. Shariati}

\address{Faculty of Mathematics and Computer Science,
Amirkabir University of Technology, 424 Hafez Avenue, 15914
Tehran, Iran.}

\email{f.shariati@aut.ac.ir}

\author[A. Pourabbas]{A. Pourabbas}
\email{arpabbas@aut.ac.ir}

\author[A. Sahami]{A. Sahami}

\address{Department of Mathematics Faculty of Basic Science, Ilam University, P.O. Box 69315-516 Ilam, Iran.}
\email{amir.sahami@aut.ac.ir}



\maketitle

\begin{abstract}
In this paper, we study the notion of Connes amenability for a class of $I\times{I}$-upper triangular matrix algebra $\A$, where $\mathcal{A}$ is a dual Banach algebra with a non-zero $wk^\ast$-continuous character and $I$ is a totally ordered set. For this purpose, we characterize the $\phi$-Connes amenability of a dual Banach algebra $\mathcal{A}$ through the existence of a specified net in $\mathcal{A}\hat{\otimes}\mathcal{A}$, where $\phi$ is a non-zero $wk^\ast$-continuous character. Using this, we show that $UP(I,\mathcal{A})$ is Connes amenable if and only if $I$ is singleton and $\mathcal{A}$ is Connes amenable. In addition, some examples of $\phi$-Connes amenable dual Banach algebras, which is not Connes amenable are given. 
\end{abstract}
\maketitle

\section{Introduction and Preliminaries}
The concept of amenability for Banach algebras was first introduced by B. E. Johnson \cite{John:72}. Let $\mathcal{A}$ be a Banach algebra and let $E$ be a Banach $\mathcal{A}$-bimodule. A bounded linear map $D:\mathcal{A}\rightarrow{E} $ is called a derivation if for every $a,b\in\mathcal{A}$, $D(ab)=a\cdot{D(b)}+D(a)\cdot{b}$. A Banach algebra $\mathcal{A}$ is called amenable if every derivation from $\mathcal{A}$ into each dual Banach $\mathcal{A}$-bimodule ${E}^{\ast}$ is inner, that is, there exists a $x\in{{E}^{\ast}}$ such that $D(a)=a\cdot{x}-x\cdot{a}$ ($a\in{\mathcal{A}}$). Let $\mathcal{A}$ be a Banach algebra. An $\mathcal{A}$-bimodule $E$ is called dual if there is a closed submodule ${E}_{\ast}$ of ${E}^{\ast}$ such that $E=(E_{\ast})^{\ast}$. The Banach algebra $\mathcal{A}$ is called dual if it is dual as a Banach $\mathcal{A}$-bimodule. A dual Banach $\mathcal{A}$-bimodule $E$ is normal, if for each $x\in{E}$ the module maps $\mathcal{A}\rightarrow{E}$; ${a}\mapsto{a}\cdot{x}$ and ${a}\mapsto{x}\cdot{a}$ are $wk^\ast$-$wk^\ast$ continuous.
 The class of dual Banach algebras was introduced by Runde \cite{Runde:2001}. The measure algebras $M(G)$ of a locally compact group $G$, the algebra of bounded operators $\mathcal{B}(E)$, for a reflexive Banach space $E$ and the second dual $\mathcal{A}^{\ast\ast}$ of Arens regular Banach algebra $\mathcal{A}$ are examples of dual Banach algebras. A suitable concept of amenability for dual Banach algebras is the Connes amenability. This notion under different name, for the first time was introduced by Johnson, Kadison, and Ringrose for von Neumann algebras \cite{John:72}. The concept of Connes amenability for the larger class of dual Banach algebras was later extended by Runde \cite{Runde:2001}. A dual Banach algebra $\mathcal{A}$ is called Connes amenable if for every normal dual Banach $\mathcal{A}$-bimodule $E$, every $wk^\ast$-continuous derivation $D:\mathcal{A}\longrightarrow{E}$ is inner. For a given dual Banach algebra $\mathcal{A}$ and a Banach $\mathcal{A}$-bimodule $E$, $\sigma{wc}(E)$ denote the set of all elements $x\in{E}$ such that the module maps $\mathcal{A}\rightarrow{E}$; ${a}\mapsto{a}\cdot{x}$ and ${a}\mapsto{x}\cdot{a}$
are $wk^\ast$-$wk$-continuous, one can see that, it is a closed submodule of $E$. If $\theta:E\longrightarrow{F}$ is a bounded $\mathcal{A}$-bimodule homomorphism, where $F$ is another Banach $\mathcal{A}$-bimodule, then $\theta(\sigma{wc}(E))\subseteq\sigma{wc}(F)$. Runde also showed that $E=\sigma{wc}(E)$ if and only if $E^{\ast}$ is a normal dual Banach $\mathcal{A}$-bimodule  \cite[Proposition 4.4]{Runde:2004}. Let $\mathcal{A}$ be a Banach algebra. The projective tensor
product $\mathcal{A}\hat{\otimes}\mathcal{A}$ is a Banach $\mathcal{A}$-bimodule with the usual
left and right operations with respect to $\mathcal{A}$. Then the map $\pi:\mathcal{A}\hat{\otimes}\mathcal{A}\longrightarrow\mathcal{A}$ defined by $\pi(a\otimes{b})=ab$ is an $\mathcal{A}$-bimodule homomorphism.  Since $\sigma{wc}(\mathcal{A}_{\ast})=\mathcal{A}_{\ast}$, the adjoint of $\pi$ maps $\mathcal{A}_{\ast}$ into $\sigma{wc}(\mathcal{A}\hat{\otimes}\mathcal{A})^{\ast}$. Therefore, $\pi^{\ast\ast}$ drops to an $\mathcal{A}$-bimodule homomorphism $\pi_{\sigma{wc}}:(\sigma{wc}(\mathcal{A}\hat{\otimes}\mathcal{A})^{\ast})^{\ast}\longrightarrow\mathcal{A}$. Any element $M\in{(\sigma{wc}(\mathcal{A}\hat{\otimes}\mathcal{A})^{\ast})^*}$ satisfying
\begin{center}
	$a\cdot{M}=M\cdot{a}\quad$ and $\quad{a}\cdot\pi_{\sigma{wc}}M=a\quad(a\in{\mathcal{A}})$,
\end{center}
is called a $\sigma{wc}$-virtual diagonal for $\mathcal{A}$. Runde showed that a dual Banach algebra $\mathcal{A}$ is Connes amenable if and only if there is a $\sigma{wc}$-virtual diagonal for $\mathcal{A}$ \cite[Theorem 4.8]{Runde:2004}.

Let $\mathcal{A}$ be a  Banach algebra and let $I$ be a totally ordered set.
 Sahami \cite{Sahami:2016} studied the notions of amenability and its related homological notions for a class of $I\times{I}$-upper triangular matrix algebra  
$$UP(I,\mathcal{A})=\set{\left[
	\begin{array}{rr} a_{i,j}  \end{array} \right]_{i,j\in I};  a_{i,j}\in A {\hbox{ and\, $a_{i,j}=0$ \,for every } }i>j }.$$

He showed that $\A$ is pseudo-contractible (amenable) if and only if $I$ is singleton and $\mathcal{A}$ is pseudo-contractible (amenable), respectively. He also studied the notions of pseudo-amenability and approximate biprojectivity of $UP(I,\mathcal{A})$. 

In this paper, we investigate the notion of Connes amenability for a class of $I\times{I}$-upper triangular matrix $UP(I,\mathcal{A})$, where $\mathcal{A}$ is a dual Banach algebra and $I$ is a totally ordered set. For this purpose, first in section 2 we study the duality of $\A$ by considering the isometric-isomorphism $\A\cong\ljs$ as Banach spaces, where $J$ is a subset of $I\times{I}$ and for every $\alpha\in{J}$, $\mathcal{A_{\alpha}}=\mathcal{A}$. In section 3 by using the fact that every Connes amenable Banach algebra is $\phi$-Connes amenable, we obtain a new characterization of  $\phi$-Connes amenability through the existence of a bounded net with a certain condition, where $\phi$ is a non-zero $wk^\ast$-continuous character. By applying latter characterization, we show that $\A$ is Connes amenable if and only if $\mathcal{A}$ is Connes amenable and $I$ is singleton. Finally in section 4 we provide some examples of $\phi$-Connes amenable dual Banach algebras, which are not Connes amenable.
\section{The duality of $\A$}
 Let $\mathcal{A}$ be a dual Banach algebra and let $I$ be a totally ordered set. Then the set of all $I\times{I}$-upper triangular matrices
   with the usual matrix operations and the norm $\parallel[a_{i,j}]_{{i,j}\in{I}}\parallel=\sum\limits _{i,j\in{I}}\parallel{a}_{i,j}\parallel<\infty$, becomes a Banach algebra. 
Before we study the duality of $\A$, we state the following Lemma:
\begin{lemma}\label{l1}
	If $\mathcal{A}$ is a dual Banach algebra with the predual $\mathcal{A}_{\ast}$ and $I$ is a non-empty set, then\\ $(\cs)^{\ast}\cong\ls$, where for every $i\in{I}$, $\mathcal{A}_{i}=\mathcal{A}$ and $\mathcal{A}_{i_\ast}=\mathcal{A}_{\ast}$.
\end{lemma}\label{l2.1}
\begin{proof}
	Let $g=(g_{\alpha})_{\alpha\in{I}}\in{\ls}$. We define $\phi_{g}: \cs\longrightarrow \mathbb{C}$ by $\phi_{g}(f)=\sum\limits_{\alpha\in{I}}g_{\alpha}(f_{\alpha})$, where $f=(f_{\alpha})_{\alpha\in{I}}\in{\cs}$. We show that $\phi_{g}$ is bounded,
	\begin{equation}\label{e2.1}
	\begin{split}
	\mid\phi_{g}(f)\mid&\leq\sum\limits_{\alpha\in{I}}\mid{g}_{\alpha}(f_{\alpha})\mid\leq\sum\limits_{\alpha\in{I}}\parallel{f}_{\alpha}\parallel\parallel{g}_{\alpha}\parallel\\
	&\leq\parallel{f}\parallel_{\infty}\sum\limits_{\alpha\in{I}}\parallel{g}_{\alpha}\parallel\leq\parallel{f}\parallel_{\infty}\parallel{g}\parallel_{1},
	\end{split}
	\end{equation}
	So $\parallel\phi_{g}\parallel\leq\parallel{g}\parallel_{1}$. Now we define $T:\ls\longrightarrow(\cs)^{\ast}$ by $T(g)=\phi_{g}$ and we show that $T$ is isometric-isomorphism. It is clear that the map $T$ is linear. Let $\phi\in{(\cs)^{\ast}}$, we show that there exists $g\in{\ls}$ such that $\phi_{g}=\phi$. Fixed $\alpha_0\in{I}$ and $\lambda_0\in{\mathcal{A}_{\ast}}$, we define  $\delta(\alpha_0;\lambda_0)\colon{I}\longrightarrow\mathcal{A}_{\ast}$ by $\delta(\alpha_0;\lambda_0)(\alpha)=\lambda_0$ whenever $\alpha=\alpha_0$ and $\delta(\alpha_0;\lambda_0)(\alpha)=0$ for every $\alpha\neq\alpha_0$. It is obvious that $\delta(\alpha_0;\lambda_0)\in \cs$ and $\parallel\delta(\alpha_0;\lambda_0)\parallel_{\infty}=\parallel\lambda_0\parallel$. Now we consider $g=(g_{\alpha})_{\alpha\in{I}}$, where $g_{\alpha}(\lambda)=\phi(\delta(\alpha;\lambda))$ for every  $\alpha\in{I}$ and $\lambda\in{\mathcal{A}_{\ast}}$. It is easy to see that $\norm{g_\alpha}\leq \norm{\phi}$ for every $\alpha\in{I}$, so $g_{\alpha}$ is a continuous linear functional on $\mathcal{A}_{\ast}$.
	
	Consider $(f_{\alpha})_{\alpha\in{I}}\in{\cs}$, we define $\eta_{\alpha}:=\dfrac{\overline{\phi(\delta(\alpha;f_{\alpha}))}}{\mid\phi(\delta(\alpha;f_{\alpha}))\mid}$, whenever $\phi(\delta(\alpha;f_{\alpha}))\neq0$ otherwise we define $\eta_{\alpha}=1$. So for every $\alpha \in I$, $\abs{\eta_{\alpha}}=1$. Let $\mathcal{F}$ be the family of all  finite subsets of $I$. Then for every $f=(f_{\alpha})_{\alpha\in{I}}\in{\cs}$ we have
	 $\norm{f}_\infty=\sup\limits_{F\in\mathcal{F}}\sum\limits_{\alpha\in{F}}(\norm{f_{\alpha}}))$. So for every $f$ in  the unit ball of $\cs$, we have
	\begin{equation*}
	\begin{split}
	\sum\limits_{\alpha\in{F}}\mid{g}_{\alpha}(f_{\alpha})\mid&=\sum\limits_{\alpha\in{F}}\mid\phi(\delta(\alpha;f_{\alpha}))\mid=\sum\limits_{\alpha\in{F}}\eta_{\alpha}\phi(\delta(\alpha;f_{\alpha}))\\&=\sum\limits_{\alpha\in{F}}\phi(\eta_{\alpha}\delta(\alpha;f_{\alpha}))=\sum\limits_{\alpha\in{F}}\phi(\delta(\alpha;\eta_{\alpha}f_{\alpha}))\\&=\phi(\sum\limits_{\alpha\in{F}}\delta(\alpha;\eta_{\alpha}f_{\alpha}))\leq\parallel\phi\parallel\parallel{f}\parallel_{\infty}\leq\norm{\phi}.
	\end{split}
	\end{equation*}
So for every ${F\in\mathcal{F}}$ we have $\sum\limits_{\alpha\in F}\parallel{g}_{\alpha}\parallel\leq\parallel\phi\parallel$ which implies that
	\begin{equation}\label{e2.2}
	\sum\limits_{\alpha\in{I}}\parallel{g}_{\alpha}\parallel\leq\parallel\phi\parallel.
	\end{equation}
	Thus $g=(g_{\alpha})_{\alpha\in{I}}\in{\ls}$. Since $\ccs$ is dense in $\cs$, first we show that $\phi_{g}=\phi$ on $\ccs$. Let $f=(f_{\alpha})_{\alpha\in{I}}\in{\ccs}$, so there exists a finite subset $F$ of $I$ such that for any $\alpha\in{I-{F}}$, $f_{\alpha}=0$. We have 
	\begin{equation}\label{e2.3}
	\begin{split}
	\phi_{g}(f)=\sum\limits_{\alpha\in{I}}g_{\alpha}(f_{\alpha})=\sum\limits_{\alpha\in{F}}\phi(\delta(\alpha;f_{\alpha}))=\phi(\sum\limits_{\alpha\in{F}}\delta(\alpha;f_{\alpha}))=\phi(f).
	\end{split}
	\end{equation}
	Now suppose that $f\in{\cs}$, so there exists a net $f_{\alpha}\in{\ccs}$ such that $f_{\alpha}\xrightarrow{\parallel.\parallel_{\infty}}f$. By (\ref{e2.3}) we have 
	\begin{equation*}
	\phi(f)=\phi(\lim\limits_{\alpha}f_{\alpha})=\lim\limits_{\alpha}\phi(f_{\alpha})=\lim\limits_{\alpha}\phi_{g}(f_{\alpha})=\phi_{g}(\lim\limits_{\alpha}f_{\alpha})=\phi_{g}(f).
	\end{equation*}
	Hence $\phi_{g}=\phi$. Now by (\ref{e2.2}) and (\ref{e2.1}) we have 
	$$\parallel{g}\parallel_{1}\leq\parallel\phi_{g}\parallel=\parallel{T(g)}\parallel\leq\parallel{g}\parallel_{1}.$$
	Therefore the map $T$ is isometry and by applying the open mapping theorem, we have
	$$\ls\cong(\cs)^{\ast}.$$
\end{proof}
 \begin{Remark}\label{r1}
	Let $\mathcal{A}$ be a dual Banach algebra and let $I$ be a totally ordered set. Consider the subset $J$ of $I\times{I}$ defined by $J=\{(i,j)\mid{i,j}\in{I},i\leq{j}\}$. So we have an isometric-isomorphism $\A\cong\ljs$ as Banach spaces, where for every $\alpha=(i,j)\in{J}$, $\mathcal{A}_{\alpha}=\mathcal{A}$. 
\end{Remark}

\begin{Theorem}\label{t2.1}
	If $\mathcal{A}$ is a dual Banach algebra with the predual $\mathcal{A}_{\ast}$ and $I$ is a totally ordered set, then $UP(I,\mathcal{A})$ is a dual Banach algebra.
\end{Theorem}
\begin{proof}
According to Remark \ref{r1} and by Lemma \ref{l1}, it is sufficient to show that $\cjs$ is a closed $\A$-submodule of $\linf$, where for every $\alpha\in{J}$, $\mathcal{A}_{\alpha}^{\ast}=\mathcal{A}^{\ast}$ and $\mathcal{A}_{\alpha_\ast}=\mathcal{A}_{\ast}$.  First we show that $\cjs$ is a closed subspace of $\linf$. Let $x_n=(\xi_\alpha^n)_{\alpha\in{J}}$ be in $\cjs$ and suppose that $x_n\longrightarrow{x}=(\xi_\alpha)_{\alpha\in{J}}$ in $\linf$. Fixed $\varepsilon>0$. For sufficiently large $N$,
\begin{equation*}
\underset{\alpha\in{J}}{\sup}\parallel\xi_\alpha^N-\xi_\alpha\parallel<\frac{\varepsilon}{2}.
\end{equation*}
Since $(\xi_\alpha^N)$ vanishes at infinity, for sufficiently large $\alpha$, we have $\parallel\xi_\alpha^N\parallel<\dfrac{\varepsilon}{2}$.
Then
\begin{equation*}
\begin{split}
\parallel{\xi}_\alpha\parallel=\parallel{\xi}_\alpha-\xi_\alpha^N+\xi_\alpha^N\parallel&\leq\parallel{\xi}_\alpha-\xi_\alpha^N\parallel+\parallel\xi_\alpha^N\parallel\\&<\frac{\varepsilon}{2}+\frac{\varepsilon}{2}=\varepsilon,
\end{split}
\end{equation*}
  for sufficiently large $\alpha$. It follows that ${x}=(\xi_\alpha)_{\alpha\in{J}}\in\cjs$. Now we show that $\cjs$ is an $\A$-module. Suppose that $X=(a_{\alpha})_{\alpha\in{J}}$ and $\Lambda=(f_{\alpha})_{\alpha\in{J}}$ are arbitrary elements in $\ljs$ and $\cjs$, respectively. For every $Y=(b_{\alpha})_{\alpha\in{J}}$ in $\ljs$, we have ${X}\cdot{\Lambda}(Y)=\Lambda(YX)$. Then by Lemma \ref{l1} and (\ref{e2.1}) we have
  \begin{equation}\label{e2.4}
  {X}\cdot{\Lambda}(Y)=\sum\limits_{\alpha\in{J}}f_{\alpha}(c_{\alpha}),
  \end{equation}
where $(c_{\alpha})_{\alpha\in{J}}=[c_{i,j}]_{{i,j}\in{I}}=YX$ with respect to matrix multiplication in $\A$.  Since $\linf$ is an $\A$-bimodule with dual actions, then we have $X\cdot{\Lambda}\in{\linf}$. We claim that $X\cdot{\Lambda}$ belongs to $\cjs$, that is, vanishes at infinity. By (\ref{e2.4}) we have 
 \begin{equation}\label{e2.5}
  \begin{split}
  {X}\cdot{\Lambda}(Y)&=\sum\limits_{i,j\in{I}}\langle{c}_{i,j},f_{i,j}\rangle=\sum\limits_{i,j\in{I}}\langle\sum\limits_{k\in{I}}b_{i,k}a_{k,j},f_{i,j}\rangle\\&=\sum\limits_{i,j\in{I}}\sum\limits_{k\in{I}}\langle{b}_{i,k}a_{k,j},f_{i,j}\rangle=\sum\limits_{i,j\in{I}}\sum\limits_{k\in{I}}\langle{b}_{i,k},a_{k,j}\cdot{f}_{i,j}\rangle,
  \end{split}
  \end{equation}
  where $i,j\in{I}$ and $i\leq{j}$. Since $\parallel{Y}\parallel_{1}<\infty$, one can see that $\underset{i,j\in{I}}{\sup}\parallel{b}_{i,j}\parallel<\infty$. Let $M=\underset{i,j\in{I}}{\sup}\parallel{b}_{i,j}\parallel$. Take a finite  subset $F$ of $I$. We have 
  \begin{equation*}
  \begin{split}
  \sum\limits_{i,j\in{F}}\sum\limits_{k\in{F}}\mid\langle{b}_{i,k},a_{k,j}\cdot{f}_{i,j}\rangle\mid&\leq\sum\limits_{i,j\in{F}}\sum\limits_{k\in{F}}\parallel{a_{k,j}}\parallel\parallel{f}_{i,j}\parallel\parallel{b}_{i,k}\parallel\\ &\leq\parallel\Lambda\parallel_{\infty}{M}\sum\limits_{i,j\in{F}}\sum\limits_{k\in{F}}\parallel{a_{k,j}}\parallel\\
  &\leq\parallel\Lambda\parallel_{\infty}{M}\sum\limits_{i,j\in{F}}\parallel{a_{i,j}}\parallel\leq\parallel\Lambda\parallel_{\infty}{M}\parallel{X}\parallel_{1}.
  \end{split}
  \end{equation*}
So $\sum\limits_{i,j\in{I}}\sum\limits_{k\in{I}}\mid\langle{b}_{i,k},a_{k,j}\cdot{f}_{i,j}\rangle\mid<\infty$.
By rearrangement series in (\ref{e2.5}), we have
\begin{equation}\label{e2.6}
{X}\cdot{\Lambda}(Y)=\sum\limits_{i,j\in{I}}\sum\limits_{k\in{I}}\langle{b}_{i,j},a_{j,k}\cdot{f}_{i,k}\rangle=\sum\limits_{i,j\in{I}}\langle{b}_{i,j},\sum\limits_{k\in{I}}a_{j,k}\cdot{f}_{i,k}\rangle.
\end{equation} 
Suppose that $X\cdot{\Lambda}=(g_{\alpha})_{\alpha\in{J}}$. By (\ref{e2.6}) for every $\alpha=(i,j)\in{J}$, we have $g_{\alpha}=g_{i,j}=\sum\limits_{k\in{I}}{a}_{j,k}\cdot{f}_{i,k}$. Fixed $\varepsilon>0$. Since $\Lambda$ vanishes at infinity, there is a $\alpha_0=(i_0,j_0)\in{J}$ such that for every $\alpha\geq\alpha_0$ we have $\parallel{f}_{\alpha}\parallel\leq\frac{\varepsilon}{\parallel{X}\parallel}$. Now for every $(i,j)\geq(i_0,j_0)$ in $J$ with product ordering, we have
 \begin{equation*}
 \parallel{g}_{i,j}\parallel\leq\sum\limits_{k\in{I}}\parallel{a}_{j,k}\parallel\parallel{f}_{i,k}\parallel\leq\frac{\varepsilon}{\parallel{X}\parallel}\sum\limits_{k\in{I}}\parallel{a}_{j,k}\parallel\leq\frac{\varepsilon}{\parallel{X}\parallel}\parallel{X}\parallel\leq\varepsilon,
 \end{equation*}
  note that in $\A$, if $j>k$, then ${a}_{j,k}=0$. Therefore $X\cdot{\Lambda}$ vanishes at infinity. 
\end{proof}
\section{Connes amenability of $\A$}
Let $\mathcal{A}$ be a dual Banach algebra and let $I$ be a totally ordered set.
  In this section  we characterize the notion of Connes amenability of $\A$. Throughout
this section 
 the set of all homomorphism from $\mathcal{A}$ into $\mathbb{C}$ is denoted by $\Delta(\mathcal{A})$ and the set of all $wk^{*}$-continuous homomorphism from $\mathcal{A}$ into $\mathbb{C}$ is denoted by $\Delta_{{wk}^{\ast}}{(\mathcal{A})}$.
For every $\varphi\in{\Delta(\mathcal{A})}$, the notion of $\varphi$-amenability for a Banach algebra was introduced by Kaniuth, Lau and Pym \cite{kan:2008}. Indeed  $\mathcal{A}$ is $\varphi$-amenable if there exists a bounded linear functional $m$ on  ${\mathcal{A}}^{\ast}$ satisfying $m(\varphi)=1$ and $m(f\cdot{a})=\varphi(a)m(f)$ for every $a\in{\mathcal{A}}$ and $f\in{\mathcal{A}^{\ast}}$. They characterized $\varphi$-amenability in different ways: 
\begin{itemize}
	\item
	Through vanishing of the cohomology groups $\mathcal{H}^{1}(\mathcal{A},X^{*})$ for certain Banach $\mathcal{A}$-bimodule $X$ \cite[Theorem 1.1]{kan:2008}.
	\item
	Through the existence of a bounded net $(u_{\alpha})$ in $\mathcal{A}$ such that$\parallel{au_{\alpha}-\varphi(a)u_{\alpha}\parallel}\longrightarrow0$ for all $a\in{\mathcal{A}}$ and $\varphi(u_{\alpha})=1$ for all $\alpha$ \cite[Theorem 1.4]{kan:2008}.
\end{itemize}
By \cite[Theorem 1.1]{kan:2008}, we conclude that every amenable Banach algebra is $\varphi$-amenable for any $\varphi\in{\Delta(\mathcal{A})}$.\\

In the sense of Connes amenability for a dual Banach algebra $\mathcal{A}$, the notion of $\varphi$-Connes amenability for $\varphi\in\Delta_{{wk}^{\ast}}{(\mathcal{A})}$, was introduced by Mahmoodi  and some characterizations were given \cite{Mahmmody:2014}. We say that $\mathcal{A}$ is $\varphi$-Connes amenable if there exists a bounded linear functional $m$ on  $\sigma{wc}({\mathcal{A}}^{\ast})$ satisfying $m(\varphi)=1$ and $m(f\cdot{a})=\varphi(a)m(f)$ for any $a\in{\mathcal{A}}$ and $f\in{\sigma{wc}({\mathcal{A}}^{\ast})}$. The concept of $\varphi$-Connes amenability was characterized through vanishing of the cohomology groups $\mathcal{H}^{1}_{wk^{*}}(\mathcal{A},E)$ for certain normal dual Banach $\mathcal{A}$-bimodule $E$.
By \cite[Theorem 2.2]{Mahmmody:2014}, we conclude that every Connes amenable Banach algebra is $\varphi$-Connes amenable for any $\varphi\in\Delta_{{wk}^{\ast}}{(\mathcal{A})}$. If $\varphi\in{\Delta_{{wk}^{\ast}}{(\mathcal{A})}}$, then one may show that, $\varphi\otimes\varphi\in{\sigma\omega{c}({\mathcal{A}}\hat{\otimes}{\mathcal{A}})^{\ast}}$, where $\varphi\otimes\varphi(a\otimes{b})=\varphi(a)\varphi(b)$ for any ${a,b}\in{\mathcal{A}}$. 

Now by inspiration of methods that used in \cite[Proposition 3.2]{Mahmmody:2016}, we characterize the notion of $\varphi$-Connes amenability through the existence of a bounded net in $\mathcal{A}\hat{\otimes}\mathcal{A}$ with certain properties.
\begin{Proposition}\label{net}
	Let $\mathcal{A}$ be a dual Banach algebra and $\varphi\in{\Delta_{{wk}^{\ast}}{(\mathcal{A})}}$. Then $\mathcal{A}$ is $\varphi$-Connes amenable if and only if  there exists a bounded net $\set{u_{\alpha}}$ in $\mathcal{A}\hat{\otimes}\mathcal{A}$ such that
	\begin{enumerate}
		\item [(i)] $a\cdot{u_{\alpha}}-\varphi(a){u_{\alpha}}{\overset{wk^\ast}{\longrightarrow}}0\qquad$in$\quad(\sigma{wc}({\mathcal{A}}\hat{\otimes}{\mathcal{A}})^{\ast})^{\ast}$.
		\item [(ii)] 
		$\langle{u_{\alpha},\varphi\otimes\varphi}\rangle\longrightarrow{1}$. 
	\end{enumerate}
\end{Proposition}
\begin{proof}
Let $\mathcal{A}$ be a	$\varphi$-Connes amenable. Then by \cite[Theorem 3.2]{Mahmmody:2014}, there exists an element $M$ in ${(\sigma{wc}({\mathcal{A}}\hat{\otimes}{\mathcal{A}})^{\ast})^{\ast}}$such that for any $a\in{\mathcal{A}}$, $a\cdot{M}=\varphi(a)M$ and $\langle{\varphi\otimes\varphi},M\rangle=1$. Since $\sigma{wc}({\mathcal{A}}\hat{\otimes}{\mathcal{A}})^{\ast}$ is a closed subspace of $(\mathcal{A}\hat{\otimes}\mathcal{A})^{\ast}$, we have a quotient map $q:(\mathcal{A}\hat{\otimes}\mathcal{A})^{\ast\ast}\longrightarrow(\sigma{wc}({\mathcal{A}}\hat{\otimes}{\mathcal{A}})^{\ast})^{\ast}$. Composing the canonical inclusion
	map $\mathcal{A}\hat{\otimes}\mathcal{A}\hookrightarrow(\mathcal{A}\hat{\otimes}\mathcal{A})^{\ast\ast}$ with $q$, we obtain a continuous $\mathcal{A}$-bimodule map $\tau:\mathcal{A}\hat{\otimes}\mathcal{A}\longrightarrow(\sigma{wc}({\mathcal{A}}\hat{\otimes}{\mathcal{A}})^{\ast})^{\ast}$ which has a $wk^{*}$-dense range. So there exists a net $(u_{\alpha})_{\alpha\in{I}}$ in $(\mathcal{A}\hat{\otimes}\mathcal{A})$ such that
	\begin{equation}\label{e3.1}
	M=wk^\ast\mbox{-}\lim_{\alpha}\tau(u_{\alpha})=wk^\ast\mbox{-}\lim_{\alpha}(\hat{u}_{\alpha})\vert_{\sigma{wc}({\mathcal{A}}\hat{\otimes}{\mathcal{A}})^{\ast}}.
	\end{equation}
By Goldstein's theorem, the net $(u_{\alpha})_{\alpha\in{I}}$ can be chosen to be a bounded net. We know that for any $T\in{{\sigma{wc}({\mathcal{A}}\hat{\otimes}{\mathcal{A}})^{\ast}}}$  and for any ${a}\in{\mathcal{A}}$,
\begin{equation*}
{T}\cdot{a}-{\varphi}(a)T\in{{\sigma{wc}({\mathcal{A}}\hat{\otimes}{\mathcal{A}})^{\ast}}}.
\end{equation*}
So
\begin{equation*}
\langle{T}\cdot{a}-{\varphi}(a)T,\hat{u}_{\alpha}\rangle\longrightarrow\langle{T}\cdot{a}-{\varphi}(a)T,M\rangle.
\end{equation*}
Thus we have
\begin{equation*}
{\langle{T},a\cdot{\hat{u}}_{\alpha}\rangle-\langle{T},\varphi(a)\hat{u}_{\alpha}\rangle}\longrightarrow{\langle{T},a\cdot{M}\rangle-\langle{T},\varphi(a)M\rangle}.
\end{equation*}
This equation is equivalent with
\begin{equation}\label{e3.2}
\langle{T},a\cdot\hat{u}_{\alpha}-\varphi(a)\hat{u}_{\alpha}\rangle\longrightarrow\langle{T},a\cdot{M}-\varphi(a)M\rangle
=0.
\end{equation}
Therefore\\

\begin{center}
$a\cdot{u_{\alpha}}-\varphi(a){u_{\alpha}}{\overset{wk^\ast}{\longrightarrow}}0\qquad${in}$\quad(\sigma{wc}({\mathcal{A}}\hat{\otimes}{\mathcal{A}})^{\ast})^{\ast}$.\\
\end{center}
On the other hand, since $\varphi\otimes\varphi\in{\sigma{wc}({\mathcal{A}}\hat{\otimes}{\mathcal{A}})^{\ast}}$, by (\ref{e3.1})
\begin{equation*}
\langle\varphi\otimes\varphi,\hat{u}_{\alpha}\rangle\longrightarrow\langle\varphi\otimes\varphi,M\rangle=1,
\end{equation*}
that is,
\begin{equation*}
\langle{u_{\alpha},\varphi\otimes\varphi}\rangle\longrightarrow{1}.
\end{equation*}
Conversely, regrad $(u_{\alpha})$ as a bounded net in $(\sigma{wc}({\mathcal{A}}\hat{\otimes}{\mathcal{A}})^{\ast})^{\ast}$. By Banach-Alaoglu theorem the bounded net $(\hat{u}_{\alpha})\vert_{\sigma{wc}({\mathcal{A}}\hat{\otimes}{\mathcal{A}})^{\ast}}$ has a $wk^\ast$-limit point. Let $$M=wk^\ast\mbox{-}\lim_{\alpha}((\hat{u}_{\alpha})\vert_{\sigma{wc}({\mathcal{A}}\hat{\otimes}{\mathcal{A}})^{\ast}}).$$
So $M\in{(\sigma{wc}({\mathcal{A}}\hat{\otimes}{\mathcal{A}})^{\ast})^{\ast}}$. By the similar argument that we apply in (\ref{e3.2}), we have
$a\cdot{M}-\varphi(a)M=0$ and $\langle\varphi\otimes\varphi,M\rangle=1$ as required. Hence by \cite[Theorem 3.2]{Mahmmody:2014} $\mathcal{A}$ is $\varphi$-Connes amenable.
\end{proof}
Now we deduce the main result of this paper.
\begin{Theorem}
Let $I$ be a totally ordered set and let $\mathcal{A}$ be a unital dual Banach algebra with $\Delta_{{wk}^{\ast}}{(\mathcal{A})}\neq\varnothing$. Then  $UP(I,\mathcal{A})$ is Connes amenable if and only if $I$ is singleton and $\mathcal{A}$ is Connes amenable.
\end{Theorem}
\begin{proof} 
Let $UP(I,\mathcal{A})$ be Connes amenable. Then by \cite[proposition 4.1]{Runde:2001}, $UP(I,\mathcal{A})$ has an identity element. But every matrix algebra with unit must be finite dimensional. So in this case $I$ is a finite set.

Assume that $I=\lbrace{i_{1},...,i_{n}}\rbrace$ and $\phi\in{\Delta_{{wk}^{\ast}}{(\mathcal{A})}}$. We define a map $\psi:{UP(I,\mathcal{A})}\longrightarrow\mathbb{C}$ by          $\left[ a_{i,j}\right] _{i,j\in{I}}\longmapsto\phi{(a_{i_{n},i_{n}})}$ for every $\left[ a_{i,j}\right] _{i,j\in{I}}\in{{UP(I,\mathcal{A})}}$.

 
 Since $\phi$ is $wk^\ast$-continuous, $\psi\in{\Delta_{{wk}^{\ast}}{(UP(I,\mathcal{A}))}}$. Now apply \cite[Theorem 2.2]{Mahmmody:2014}, one can see that  ${UP(I,\mathcal{A})}$ is $\psi$-Connes amenable. Using Proposition \ref{net}, there exists a bounded net $(u_{\alpha})\subseteq{\A}\hat{\otimes}{\A}$ such that
\begin{equation}\label{3.4}
a\cdot{\hat{u}_{\alpha}}\vert_{\sigma{wc}({\A}\hat{\otimes}{\A})^{\ast}}-\psi{(a)}{\hat{u}_{\alpha}}\vert_{\sigma{wc}({\A}\hat{\otimes}{\A})^{\ast}}{\overset{wk^\ast}{\longrightarrow}}0\qquad({a}\in{\A})
\end{equation}
and
\begin{equation}\label{3.5}
\langle{u_{\alpha},\psi\otimes\psi}\rangle\longrightarrow{1},
\end{equation}
where  $\psi\otimes\psi\in{\sigma{wc}({\A}\hat{\otimes}{\A})^{\ast}}$ and $\psi\otimes\psi(a\otimes{b})=\psi(a)\psi(b)$ for every ${a,b}\in{\A}$. 

It is well known that the map $\pi_{\sigma{wc}}:{(\sigma{wc}({\A}\hat{\otimes}{\A})^{\ast})^{\ast}}\longrightarrow\A$ is $wk^\ast$-continuous. So by (\ref{3.4}) we have
\begin{equation*}
a\cdot\pi_{\sigma{wc}}({\hat{u}_{\alpha}}\vert_{\sigma{wc}({\A}\hat{\otimes}{\A})^{\ast}})-\psi{(a)}\pi_{\sigma{wc}}({\hat{u}_{\alpha}}\vert_{\sigma{wc}({\A}\hat{\otimes}{\A})^{\ast}}){\overset{wk^\ast}{\longrightarrow}}0\qquad({a}\in{\A}).
\end{equation*}
Let $\pi_{\sigma{wc}}({\hat{u}_{\alpha}}\vert_{\sigma{wc}({\A}\hat{\otimes}{\A})^{\ast}})=m_{\alpha}$. Then $(m_{\alpha})$ is a net in $\A$  that
satisfies
\begin{equation*}
a{m_{\alpha}}-\psi(a){m_{\alpha}}{\overset{wk^\ast}{\longrightarrow}}0\qquad({a}\in{\A}).
\end{equation*}
On the other hand for every  ${f}\in\A_{\ast}$ we have
\begin{equation*}
\begin{split}
\langle{f},\pi_{\sigma{wc}}({\hat{u}_{\alpha}}\vert_{\sigma{wc}({\A}\hat{\otimes}{\A})^{\ast}})\rangle&=\langle\pi^{\ast}\vert_{\A_{\ast}}(f),{\hat{u}_{\alpha}}\vert_{\sigma{wc}({\A}\hat{\otimes}{\A})^{\ast}}\rangle\\&=\langle\pi^{\ast}(f),{\hat{u}_{\alpha}}\vert_{\sigma{wc}({\A}\hat{\otimes}{\A})^{\ast}}\rangle\\&=\langle\pi^{\ast}(f),\hat{u}_{\alpha}\rangle 
=\langle{u}_{\alpha},\pi^{\ast}(f)\rangle\\&=\langle\pi(u_{\alpha}),f\rangle,
\end{split}
\end{equation*}
so
\begin{equation}\label{3.6}
m_{\alpha}=\pi_{\sigma{wc}}({\hat{u}_{\alpha}}\vert_{\sigma{wc}({\A}\hat{\otimes}{\A})^{\ast}})=\pi(u_{\alpha}).
\end{equation}
\\Fixed $\alpha$. Since $u_{\alpha}\in{{\A}\hat{\otimes}{\A}}$, there are $b^{\alpha}_{k}$ and ${c}^{\alpha}_{k}$ in $\A$ such that $u_{\alpha}=\sum\limits_{k=1}^{\infty}b^{\alpha}_{k}\otimes{c}^{\alpha}_{k}$. So by (\ref{3.5}), we have
\begin{equation*}
\begin{split}
\psi(\pi(u_{\alpha}))&=\psi(\pi(\sum\limits_{k=1}^{\infty}b^{\alpha}_{k}\otimes{c}^{\alpha}_{k}))=\psi(\sum\limits_{k=1}^{\infty}b^{\alpha}_{k}{c}^{\alpha}_{k})\\
&=\sum\limits_{k=1}^{\infty}\psi(b^{\alpha}_{k})\psi({c}^{\alpha}_{k})=\psi\otimes\psi(u_{\alpha})\longrightarrow1,
\end{split}
\end{equation*}
therefore by (\ref{3.6}), $\psi(m_{\alpha})\longrightarrow1$. 
Let $L=\lbrace{[a_{i,j}]}\in\A\mid{a_{i,j}=0},\quad\forall{j}\neq{i_{n}}\rbrace$. Since $I$ is a finite set, it is easy to see that $L$ is a $wk^\ast$-closed ideal in $\A$. By definition of the map $\psi$, we have $\psi\vert_{L}\neq0$. So there exists $\lambda\in{L}$ such that $\psi(\lambda)\neq0$, by replacing   $\frac{\lambda}{\psi(\lambda)}$ if necessary,  we may assume that $\psi(\lambda)=1$. Let $n_{\alpha}=m_{\alpha}\lambda$. Then $n_{\alpha}$ is a net in $L$. Since  $l{m_{\alpha}}-\psi(l)m_{\alpha}{\overset{wk^\ast}{\longrightarrow}}0$  for any $l\in{L}$ and since the multiplication in $\A$ is separately $wk^\ast$-continuous \cite[Exercise 4.4.1]{Runde:2002}, we have
\begin{equation}\label{3.7}
l{n_{\alpha}}-\psi(l)n_{\alpha}=(l{m_{\alpha}}-\psi(l)m_{\alpha})\lambda{\overset{wk^\ast}{\longrightarrow}}0,
\end{equation} for every $l\in{L}$ and also 
\begin{equation*}
\psi(n_{\alpha})=\psi(m_{\alpha})\psi(\lambda)=\psi(m_{\alpha})\longrightarrow1.
\end{equation*}
Now suppose that $\lvert{I}\rvert>1$. Set 
$n_{\alpha}=\left(\begin{array}{ccc} 0&\cdots&x_{1}^\alpha\\
\colon&\cdots&\colon\\
0&\cdots&x_{n}^\alpha
\end{array}
\right)$, where $x_{1}^\alpha,\ldots,{x}_{n}^\alpha\in{\mathcal{A}}$. Consider $l=\left(\begin{array}{ccc} 0&\cdots&l_{1}\\
\colon&\cdots&\colon\\
0&\cdots&l_{n}
\end{array}
\right)$, where $l_{1},\ldots,l_{n}\in\mathcal{A}$ and $\phi(l_{1})=\ldots=\phi(l_{n-1})=1$ but $\psi(l)=\phi(l_{n})=0$. So we have $ln_{\alpha}=\left(\begin{array}{ccc} 0&\cdots&l_{1}x_{n}^\alpha\\
\colon&\cdots&\colon\\
0&\cdots&l_{n}x_{n}^\alpha
\end{array}
\right)$. By (\ref{3.7}), we have $ln_{\alpha}{\overset{wk^\ast}{\longrightarrow}}0$. Since $I$ is a finite set, it is easy to see that $l_{1}x_{n}^\alpha{\overset{wk^\ast}{\longrightarrow}}0$. Since $\phi$ is $wk^\ast$-continuous, $\phi(l_{1}x_{n}^\alpha)\longrightarrow0$. So $\phi(l_1)\phi(x_{n}^\alpha)\longrightarrow0$. Since $\phi(l_1)=1$, $\phi(x_{n}^\alpha)\longrightarrow0$, which is a contradiction with $\phi(x_{n}^\alpha)=\psi(n_{\alpha})\longrightarrow1$. Thus $\lvert{I}\rvert=1$.
 
Converse is clear.
\end{proof}
\section{Examples} 
 Here we give two examples of $\phi$-Connes amenable dual Banach algebras, which are not Connes amenable.
\begin{Example}
	Let $\mathcal{H}$ be a Hilbert space with $\hbox{dim}\,\mathcal{H}>1$. Suppose that $\phi$ is a non-zero linear functional on $\mathcal{H}$ with $\parallel\phi\parallel\leq1$. Define $a\ast{b}=\phi(a)b$ for every $a,b\in\mathcal{H}$. One can easily show that $(\mathcal{H},\ast)$ is a Banach algebra and $\Delta(\mathcal{H})=\{\phi\}$. We claim that $(\mathcal{H},\ast)$ is a dual Banach algebra. By \cite[Exercise 4.4.1]{Runde:2002}, it is sufficient to show that the multiplication $\ast$ is separately $wk^\ast$-continuous. Let $(a_{\alpha})_{\alpha\in{I}}$ be a net in $\mathcal{H}$ such that $a_{\alpha}\xrightarrow{wk^\ast}a$ and let $b\in{\mathcal{H}}$. So $$b\ast{a_{\alpha}}=\phi({b})a_{\alpha}\xrightarrow{wk^\ast}\phi({b})a=b\ast{a}.$$ Since $\mathcal{H}^{\ast\ast}=\mathcal{H}$, $a_{\alpha}(\phi)\longrightarrow{a}(\phi)$. So $\phi(a_{\alpha})\longrightarrow\phi(a)$. Hence
	$$a_{\alpha}\ast{b}=\phi(a_{\alpha})b\longrightarrow\phi(a)b=a\ast{b}.$$ So $a_{\alpha}\ast{b}\xrightarrow{wk^\ast}a\ast{b}$. Thus $(\mathcal{H},\ast)$ is a dual Banach algebra. Already we have shown that $\phi$ is a $wk^\ast$-continuous character on $\mathcal{H}^{\ast}$. Pick $a_0$ in ${\mathcal{H}}$ such that $\phi(a_0)=1$. So $a\ast{a_0}=\phi(a){a_0}$ and $\phi(a_0)=1$ for every $a\in{\mathcal{H}}$. Thus ${\mathcal{H}}$ is $\phi$-amenable. Since $\mathcal{H}^{\ast\ast}=\mathcal{H}$ is a normal dual Banach $\mathcal{H}$-bimodule, by \cite[Proposition 4.4]{Runde:2004}, $\sigma{wc}(\mathcal{H}^{\ast})=\mathcal{H}^\ast=\mathcal{H}$. So ${a_0}\in{(\sigma{wc}(\mathcal{H}^{\ast}))^\ast}$ such that $a_0(\phi)=1$ and $$a_0(f\cdot{a})=f\cdot{a}(a_0)=f(a\ast{a_0})=f(\phi(a){a_0})=\phi(a){a_0}(f),$$ for every $a\in{\mathcal{H}}$ and $f\in{\sigma{wc}(\mathcal{H}^{\ast})}$. So ${\mathcal{H}}$ is $\phi$-Connes amenable. We assume conversely that ${\mathcal{H}}$ is Connes amenable. Then ${\mathcal{H}}$ has an identity, say $E$. So for every $a\in{\mathcal{H}}$, $\phi(a)E=a\ast{E}=E\ast{a}=a$. It follows that $a=\phi(a)E$ for every $a\in{\mathcal{H}}$. So $\hbox{dim}\,\mathcal{H}=1$, which is a contradiction.
\end{Example}
\begin{Example}
Set $\mathcal{A}=\left(\begin{array}{cc} \mathbb{C}&\mathbb{C}\\
	0&0\\
\end{array}
\right)$. With the usual matrix multiplication and $\ell^{1}$-norm, $\mathcal{A}$ is a Banach algebra. Since $\mathbb{C}$ is a dual Banach algebra, $\mathcal{A}$ is a dual Banach algebra. We define a map $\phi:\mathcal{A}\longrightarrow\mathbb{C}$ by $\phi\left(\begin{array}{cc} x&y\\
0&0\\
\end{array}
\right)=x$. It is clear that $\phi$ is linear and multiplicative. Suppose that $X_{\alpha}=\left(\begin{array}{cc} x_\alpha&y_\alpha\\
0&0\\
\end{array}
\right)$ and $X=\left(\begin{array}{cc} x&y\\
0&0\\
\end{array}
\right)$ are elements in $\mathcal{A}$ such that $X_{\alpha}\xrightarrow{wk^\ast}X$, it is easy to see that $x_\alpha\rightarrow{x}$, thus $\phi$ is $wk^\ast$-continuous and
   $\phi\in{\Delta_{{wk}^{\ast}}{(\mathcal{A})}}$.   Now  we show that $\mathcal{A}$ is not Connes amenable. If $\mathcal{A}$ is Connes amenable, then by applying \cite[Proposition 4.1]{Runde:2001}, $\mathcal{A}$ has an identity say $E=\left(\begin{array}{cc} x_{0}&y_{0}\\
  0&0\\
  \end{array}\right)$, where $x_{0},y_{0}\in{\mathbb{C}}$. 
  Since $\phi$ is a multiplicative functional, $\phi(E)=1$. So $x_{0}=1$. For every $a,b\in{\mathbb{C}}$, we have 
\begin{equation}\label{e3.8}
\left(\begin{array}{cc} a&b\\
0&0\\
\end{array}
\right)=\left(\begin{array}{cc} a&b\\
0&0\\
\end{array}
\right)\left(\begin{array}{cc} 1&y_{0}\\
0&0\\
\end{array}
\right)=\left(\begin{array}{cc} a&ay_{0}\\
0&0\\
\end{array}
\right),
\end{equation}
which implies that $ay_{0}=b$ for every $a,b\in{\mathbb{C}}$, which is a contradiction.
Hence $\mathcal{A}$ is not Connes amenable.  Next we show that $\mathcal{A}$ is $\phi$-Connes amenable. Let $u=\left(\begin{array}{cc} 1&1\\
0&0\\
\end{array}
\right)\otimes\left(\begin{array}{cc} 1&1\\
0&0\\
\end{array}
\right)\in{\mathcal{A}\hat{\otimes}\mathcal{A}}$. Since $\mathcal{A}\hat{\otimes}\mathcal{A}$ embeds in $(\sigma\omega{c}(\mathcal{A}\hat{\otimes}\mathcal{A})^{\ast})^{\ast}$, we may assume that $u$ is in $(\sigma\omega{c}(\mathcal{A}\hat{\otimes}\mathcal{A})^{\ast})^{\ast}$. Now for every $a,b\in{\mathbb{C}}$, we have
\begin{equation*}
\begin{split}
\left(\begin{array}{cc} a&b\\
0&0\\
\end{array}
\right)\cdot{u}&=\left(\begin{array}{cc} a&b\\
0&0\\
\end{array}
\right)\left(\begin{array}{cc} 1&1\\
0&0\\
\end{array}
\right)\otimes\left(\begin{array}{cc} 1&1\\
0&0\\
\end{array}
\right)=\left(\begin{array}{cc} a&a\\
0&0\\
\end{array}
\right)\otimes\left(\begin{array}{cc} 1&1\\
0&0\\
\end{array}
\right)\\
&=a\left(\begin{array}{cc} 1&1\\
0&0\\
\end{array}
\right)\otimes\left(\begin{array}{cc} 1&1\\
0&0\\
\end{array}
\right)=\phi\left(\begin{array}{cc} a&b\\
0&0\\
\end{array}
\right)\left(\begin{array}{cc} 1&1\\
0&0\\
\end{array}
\right)\otimes\left(\begin{array}{cc} 1&1\\
0&0\\
\end{array}
\right)\\
&=\phi\left(\begin{array}{cc} a&b\\
0&0\\
\end{array}
\right)u.
\end{split}
\end{equation*}
and also
\begin{equation*}
\begin{split}
\langle{u},\phi\otimes\phi\rangle&=\phi\left(\begin{array}{cc} 1&1\\
0&0\\
\end{array}
\right)\phi\left(\begin{array}{cc} 1&1\\
0&0\\
\end{array}
\right)\\
&=1\times1=1.
\end{split}
\end{equation*}
Now by  Proposition 3.1, $\mathcal{A}$ is $\phi$-Connes amenable.
\end{Example}

\begin{small}
	
\end{small}

\end{document}